\documentclass[reqno,11pt]{amsart}
\usepackage{amsmath,amsfonts,amscd,amssymb,latexsym}
\usepackage{amsrefs,hyperref}
\usepackage{comment}
\usepackage{soul}
\usepackage{tikz}
\usepackage{marginnote}
\usepackage{anysize}
\usepackage{bbm}

\numberwithin{equation}{section}

\theoremstyle{plain}
\newtheorem{theorem}[subsection]{Theorem}
\newtheorem{lemma}[subsection]{Lemma}
\newtheorem{corollary}[subsection]{Corollary}
\newtheorem{proposition}[subsection]{Proposition}

\theoremstyle{definition}
\newtheorem{definition}[subsection]{Definition}

\newtheorem{assumption}[subsection]{Assumption}

\theoremstyle{remark}
\newtheorem{remark}[subsection]{Remark}

\newcommand{\ind}{\operatorname{ind}}
\newcommand{\RR}{\mathbb{R}}
\newcommand{\ZZ}{\mathbb{Z}}

\newcommand{\BB}{\mathbf{B}}
\newcommand{\mo}{\rm mod}
\newcommand{\upper}{\uppercase\expandafter}

\begin{document}

\title{Cobordism Invariance of the Index of Callias-Type Operators}

\author{Maxim Braverman${}^\dag$}
\address{Department of Mathematics,
Northeastern University,
Boston, MA 02115,
USA}

\email{maximbraverman@neu.edu}
\urladdr{www.math.neu.edu/~braverman/}

\author{Pengshuai Shi}
\address{Department of Mathematics,
Northeastern University,
Boston, MA 02115,
USA}

\email{shi.pe@husky.neu.edu}

\thanks{${}^\dag$Supported in part by the NSF grant DMS-1005888.}


\begin{abstract}
We introduce a notion of cobordism of Callias-type operators over complete Riemannian manifolds and prove that the index is preserved by such a cobordism. As an application we prove a gluing formula for Callias-type index. In particular, a usual index of an elliptic operator on a compact manifold can be computed as a sum of indexes of Callias-type operators on two non-compact, but topologically simpler manifolds. As another application we give a new proof of the relative index theorem for Callias-type operators, which also leads to a new proof of the Callias index theorem. 
\end{abstract}
\maketitle

\maketitle

\section{Introduction}\label{S:intro}

The study of the index of a Dirac-type operator with potential $B=D+\Phi$ on a complete Riemannian manifold $M$ was initiated by Callias, \cite{Callias78}, and further studied by many authors, cf, for example, \cite{BottSeeley78}, \cite{BruningMoscovici}, \cite{Anghel93-2}, \cite{Bunke}. A celebrated Callias-type index theorem discovered by these authors in different forms states that the index of a Callias-type operator can be computed as an index of a certain operator induced by it on a compact hypersurface. Recently the interest in Callias-type index theory was revived  partially in a relation to the study of the moduli space of monopoles over non-compact manifolds. Several generalizations and new applications of the Callias-type index theorem were obtained, \cite{Kottke11}, \cite{CarvalhoNistor14}, \cite{Wimmer14}, \cite{Kottke15}.

In this paper we define a class of cobordisms between Callias-type operators and show that the Callias-type index is preserved by this class of cobordisms. The proof of this theorem is similar to the proof of the cobordism invariance of the index on a compact manifold, given in \cite{Br-cobinv}, but a more careful analysis is needed. We also present several applications of this result. 

Suppose $\Sigma\subset M$ is a compact hypersurface, such that $M\backslash\Sigma$ is a disjoint union of two submanifolds $M_1$ and $M_2$. We assume that $\Sigma$ lies outside of the {\em essential support} of the potential $\Phi$. Roughly that means that the restriction of $\Phi^2$ to $\Sigma$ is strictly positive, cf. Definition~\ref{D:calliasopr} for a more precise definition.  We endow $M_1$ and $M_2$ with complete Riemannian metrics and denote by $B_1$ and $B_2$ the operators induced by $B$ on $M_1$ and $M_2$ respectively. The gluing formula states that 
\[
	\ind(B)\ = \ \ind(B_1)\ +\ \ind(B_2).
\]
This formula should be compared with a similar gluing formula for equivariant Dirac operator with potential obtained in \cite{Br-index}. We remark that the gluing formula is useful and non-trivial even in the case when the original manifold $M$ is compact. The proof of the gluing formula is obtained by constructing a cobordism between $M$ and $M_1\sqcup{}M_2$. 

As a second application of the cobordism invariance of the Callias index, we prove a version of the relative index theorem for Callias-type operator. The relative index theorem was first proven by Gromov and Lawson \cite{GL}. Since then many different reformulations and many different proofs were suggested by different authors. In this paper we prove a version of the relative index theorem suggested by \cite{Bunke}. Notice, that Anghel \cite{Anghel93-2} used the relative index theorem to prove a Callias-type index theorem. Hence, our new proof of the relative index theorem leads to a new proof of the Callias-index theorem, based on the cobordism invariance of the Callias-index. 

\section{The Main Results}\label{S:mainresults}

\subsection{Callias-type operators}\label{SS:calliasopr}

Let $(M,g)$ be a complete Riemannian manifold (possibly with boundary) with metric $g$. $M$ is endowed with a Hermitian vector bundle $E$. We denote by $C_0^\infty(M,E)$ the space of smooth sections of $E$ with compact support, and by $L^2(M,E)$ the Hilbert space of square-integrable sections of $E$ which is the completion of $C_0^\infty(M,E)$ with respect to the norm $\|\cdot\|$ induced by the $L^2$-inner product
\begin{equation}\label{E:L2innerproduct}
	(s_1,s_2)\ =\ \int_M\langle s_1(x),s_2(x)\rangle_{E_x}d{\rm vol}(x),
\end{equation}
where $\langle\cdot,\cdot\rangle_{E_x}$ denotes the fiberwise inner product and $d{\rm vol}(x)$ is the canonical volume form induced by the metric $g$.

Let $D:C_0^\infty(M,E)\to C_0^\infty(M,E)$ be a first-order formally self-adjoint elliptic differential operator and let $\Phi\in{\rm End}(E)$ be a self-adjoint bundle map. Then $D+\Phi$ is a first-order elliptic operator, and
\begin{equation}\label{E:calliasind}
(D+\Phi)^2\ =\ D^2+\Phi^2+[D,\Phi]_+,
\end{equation}
where
\[
	[D,\Phi]_+\ :=\ D\Phi\ +\ \Phi D
\]
is the anti-commutator of the operators $D$ and $\Phi$.

\begin{definition}\label{D:calliasopr}
We say that $D+\Phi$ is a (generalized) \emph{Callias-type operator} if
\begin{enumerate}
\item $[D,\Phi]_+$ is a zeroth order differential operator, i.e. a bundle map;
\item there is a compact subset $K\Subset M$ and a constant $c>0$ such that
\[
	|(\Phi^2+[D,\Phi]_+)(x)|\ \ge\  c
\]
for all $x\in M\setminus K$. Here $|(\Phi^2+[D,\Phi]_+)(x)|$ denotes the operator norm of the linear map $(\Phi^2+[D,\Phi]_+)(x):E_x\to E_x$. In this case, the compact set $K$ is called the \emph{essential support} of $D+\Phi$.
\end{enumerate}
\end{definition}

\subsection{The index of a $\ZZ_2$-graded Callias-type operator}\label{SS:indexofCallias}

Suppose that $E=E^+\oplus E^-$ is a $\ZZ_2$-graded vector bundle and that $D$ and $\Phi$ are odd with respect to this grading. This means that with respect to this decomposition we have 
\[
D\ =\ \begin{bmatrix}
0 & D^- \\
D^+ & 0
\end{bmatrix}, \qquad \Phi\ =\ \begin{bmatrix}
0 & \Phi^- \\
\Phi^+ & 0
\end{bmatrix}.
\]

From now on, we suppose that $M$ has no boundary and $D$ satisfies the following assumption.

\begin{assumption}\label{A:leadingsymbol}
There exists a constant $k>0$ such that
\begin{equation}\label{E:leadingsymbol}
0\ <\ |\sigma(D)(x,\xi)|\ \le\ k\|\xi\|,\qquad\text{for all }x\in M,\;\xi\in T_x^*M\setminus\{0\},
\end{equation}
where $\|\xi\|$ denotes the length of $\xi$ defined by the metric $g$, $\sigma(D)(x,\xi):E_x^\pm\to E_x^\mp$ is the leading symbol of $D$.
\end{assumption}

An interesting class of examples of operators satisfying \eqref{E:leadingsymbol} is given by Dirac-type operators.

By \cite[Theorem 1.17]{GL}, \eqref{E:leadingsymbol} implies that $D$ and $D+\Phi$ are essentially self-adjoint operators with initial domain $C_0^\infty(M,E)$. We view $D+\Phi$ as an unbounded operator on  $L^2(M,E)$. By a slight abuse of notation we also denote by $D+\Phi$ the closure of $D+\Phi$. Let  $\|\cdot\|$ denote the norm on $L^2(M,E)$ induced by \eqref{E:L2innerproduct}.

\begin{remark}\label{R:invinf}
It's easy to see from \eqref{E:calliasind} that a Callias-type operator $D+\Phi$ satisfying Assumption \ref{A:leadingsymbol} is \emph{invertible at infinity}, i.e.,
\begin{equation}\label{E:invinf}
\|(D+\Phi)s\|\ge \sqrt{c}\|s\|,\qquad\text{for all }s\in L^2(M,E),\;{\rm supp}(s)\cap K=\emptyset.
\end{equation}
\end{remark}

It follows from \cite[Theorem 2.1]{Anghel93} and Remark \ref{R:invinf} that
\begin{lemma}\label{L:Callias-invertible}
If Assumption~\ref{A:leadingsymbol} is satisfied, then a Callias-type operator $D+\Phi$ is Fredholm.
\end{lemma}

Thus $\ker(D+\Phi)=\ker(D^++\Phi^+)\oplus\ker(D^-+\Phi^-)\subset L^2(M,E)$ is finite-dimensional, and the index,
\begin{equation}\label{E:index}
\ind(D+\Phi)\ :=\ \dim\ker(D^++\Phi^+)-\dim\ker(D^-+\Phi^-)
\end{equation}
is well-defined.

\subsection{Cobordism of Callias-type operators}\label{SS:cobordism}

We now introduce a class of non-compact cobordisms similar to those considered in \cite{GGK96,GGK-book,Br-index,BrCano14}. One of the main results of this paper is that the index \eqref{E:index} is preserved by this class of cobordisms.

\begin{definition}\label{D:calliascob}
Suppose $(M_1,E_1,D_1+\Phi_1)$ and $(M_2,E_2,D_2+\Phi_2)$ are two triples which are described as above. $(W,F,\tilde D+\tilde\Phi)$ is a \emph{cobordism} between them if
\begin{enumerate}
\item $W$ is a complete manifold with boundary $\partial{W}$ and there is an open neighborhood $U$ of $\partial{W}$ and a metric-preserving diffeomorphism
    \begin{equation}\label{E:diffeo}
    \phi: (M_1\times(-\epsilon,0])\sqcup(M_2\times[0,\epsilon)) \ \to \ U.
    \end{equation}
    In particular, $\partial{W}$ is diffeomorphic to the disjoint union $M_1\sqcup M_2$.
\item $F$ is a vector bundle (may not be graded) over $W$, whose restriction to $U$ is isomorphic to the lift of $E_1$ and $E_2$ over $(M_1\times(-\epsilon,0])\sqcup(M_2\times[0,\epsilon))$;
\item $\tilde D+\tilde\Phi:C_0^\infty(W,F)\to C_0^\infty(W,F)$ is a Callias-type operator {with $\tilde D$ satisfying Assumption \ref{A:leadingsymbol}}, and takes the form
    \begin{equation}\label{E:tildeD+Phi}
    \tilde D+\tilde\Phi\ =\ D_i+\gamma\partial_t+\Phi_i
    \end{equation}
    on $U$, where $t$ is the normal coordinate and $\gamma|_{E_i^\pm}=\pm\sqrt{-1}$, $i=1,2$.
\end{enumerate}

If there exists a cobordism between $(M_1,E_1,D_1+\Phi_1)$ and $(M_2,E_2,D_2+\Phi_2)$ then the operators $D_1+\Phi_1$ and $D_2+\Phi_2$ are called \emph{cobordant}.
\end{definition}

\begin{remark}
If $M_2$ is the empty manifold, then $(W,F,\tilde D+\tilde\Phi)$ is a \emph{null-cobordism} of $(M_1,E_1,D_1+\Phi_1)$. In this case the operator $D_1+\Phi_1$ is called \emph{null-cobordant}.
\end{remark}

\begin{remark}\label{R:M=M1cupM2}
Let $E_2^{\rm op}$ denote the vector bundle $E_2$ with opposite grading, namely $E_2^{\rm op\pm}=E_2^{\mp}$.  Consider the vector bundle $E$ over $M=M_1\sqcup M_2$ induced by $E_1$ and $E_2^{\rm op}$. Let $D+\Phi:C_0^\infty(M,E^\pm)\to C_0^\infty(M,E^\mp)$ be the operator such that $D|_{M_i}=D_i$, $\Phi|_{M_i}=\Phi_i$, $i=1,2$. Then $(W,F,\tilde D+\tilde\Phi)$ makes $D+\Phi$ null-cobordant.
\end{remark}

\subsection{Cobordism invariance of the index}\label{SS:cobinv}

We now formulate the main result of the paper:	

\begin{theorem}\label{T:cobinv}
Let $D_1+\Phi_1$ and $D_2+\Phi_2$ be cobordant Callias-type operators. Then 
\[
	\ind(D_1+\Phi_1)\ =\ \ind(D_2+\Phi_2).
\]
\end{theorem}

By Remark \ref{R:M=M1cupM2}, this is equivalent to the following

\begin{theorem}\label{T:cobinv2}
The index of a null-cobordant Callias-type operator $D+\Phi$ is equal to zero.
\end{theorem}

\subsection{An outline of the proof of Theorem \ref{T:cobinv2}}\label{SS:outline}

Sections \ref{S:indbadelta}-\ref{S:pf-thm} deal with the proof of Theorem \ref{T:cobinv2}. We use the method of  \cite{Br-cobinv}, \cite{Br-cobtr} with necessary modifications. 

Suppose $(W,F,\tilde D+\tilde\Phi)$ is a null-cobordism of $(M,E,D+\Phi)$. In Section \ref{S:indbadelta}, we denote by  $\tilde W$ the manifold obtained from $W$ by attaching a semi-infinite cylinder $M\times[0,\infty)$. Then for small enough number $\delta>0$ we construct a family of Fredholm operators $\BB_{a,\delta}$ on $\tilde{W}$ whose index is independent of $a\in \RR$. 

An easy computation, cf. Lemma~\ref{L:ind=0}, shows that for $a\ll 0$ the operator $\BB_{a,\delta}^2>0$. Hence, its index is equal to 0. Hence, 
\begin{equation}\label{E:ideaindBa=0}
	\ind\BB_{a,\delta}\ =\ 0, \qquad\text{for all }a\in \RR.
\end{equation} 

In Sections~\ref{S:model} and \ref{S:pf-thm} we study the operator $\BB_{a,\delta}$ for $a\gg0$. It turns out that the sections in the kernel of this operator are concentrated on the cylinder  $M\times[0,\infty)$ near the hypersurface  $M\times\{a\}$. Then we construct an operator $\BB_\delta^{\mo}$ on the cylinder $M\times\RR$, whose restriction to a neighborhood of $M\times\{0\}$ is very close to the restriction of $\BB_{a,\delta}$ to a neighborhood of $M\times\{a\}$. In a certain sense, $\BB_\delta^{\mo}$ is the limit of $\BB_{a,\delta}$ as $a\to\infty$.  We refer to $\BB_\delta^{\mo}$ as the {\em model operator} for $\BB_{a,\delta}$. In Lemma~\ref{L:keriso} we compute the kernel of 
\begin{equation}\label{E:ideaBamod=Callias}
	\ind\BB_\delta^{\mo}\ =\ \ind(D+\Phi).
\end{equation}	
Finally, in Proposition~\ref{P:n=dker} we show that 
\begin{equation}\label{E:ideaBamod=indBa}
	\ind\BB_\delta^{\mo}\ =\ \ind \BB_{a,\delta}.
\end{equation}
Theorem \ref{T:cobinv2}  follows immediately from  \eqref{E:ideaindBa=0}, \eqref{E:ideaBamod=Callias}, and \eqref{E:ideaBamod=indBa}.

\subsection{The gluing formula}\label{SS:gluing-intro}

As a first application of Theorem~\ref{T:cobinv} we prove the gluing formula, cf.  Section \ref{S:gluing}.

Suppose that $(M,E,D+\Phi)$ is as in Subsection~\ref{SS:indexofCallias} and that $\Sigma$ is a hypersurface in $M$. Under certain conditions (cf. Assumption \ref{A:gluing}), if one cuts $M$ along $\Sigma$ and converts it to a complete manifold without boundary by rescaling the metric, one gets a new triple $(M_\Sigma,E_\Sigma,D_\Sigma+\Phi_\Sigma)$, with $D_\Sigma+\Phi_\Sigma$ being a Callias-type and, hence, a  Fredholm operator. Then the gluing formula asserts that

\begin{theorem}
The operators $D+\Phi$ and $D_\Sigma+\Phi_\Sigma$ are cobordant. In particular,
\[
\ind(D+\Phi)=\ind(D_\Sigma+\Phi_\Sigma).
\]
\end{theorem}

If $M$ is partitioned into two relatively open submanifolds $M_1$ and $M_2$ by $\Sigma$, namely, $M=M_1\cup\Sigma\cup M_2$, then the complete metric on $M_\Sigma$ induces complete Riemannian metrics on $M_1$ and $M_2$. Let $E_i,D_i,\Phi_i$ denote the restrictions of the graded vector bundle $E_\Sigma$ and operators $D_\Sigma,\Phi_\Sigma$ to $M_i$ ($i=1,2$). The above theorem implies the additivity of the index (cf. Corollary \ref{C:additivity}).

\subsection{The relative index theorem}\label{SS:relindthm-intro}

Section \ref{S:relindthm} is occupied with the second application of Theorem~\ref{T:cobinv}, which is a new proof of the well-known relative index theorem for Callias-type operators.

Consider two triples $(M_j,E_j,D_j+\Phi_j)$ as before $(j=1,2)$. Suppose $M_j'\cup_{\Sigma_j}M_j''$ are partitions of $M_j$ into relatively open submanifolds, where $\Sigma_j$ are compact hypersurfaces. 

Suppose there exist tubular neighborhoods $U(\Sigma_j)$ of $\Sigma_j$. We assume isomorphisms of structures between $\Sigma_1$ and $\Sigma_2$, $U(\Sigma_1)$ and $U(\Sigma_2)$, $E_1|_{U(\Sigma_1)}$ and $E_2|_{U(\Sigma_2)}$. We also assume that $\Phi_j$ are invertible on $U(\Sigma_j)$, and that  $D_1,\Phi_1$ coincide with $D_2,\Phi_2$ on $U(\Sigma_1)\simeq U(\Sigma_2)$ (cf. Assumption \ref{A:relindthm}). Then we can cut  $M_j$ along $\Sigma_j$ and use the isomorphism map to glue the pieces together interchanging $M_1''$ and $M_2''$. In this way we obtain the manifolds 
\[
M_3\ :=\ M_1'\cup_\Sigma M_2'',\qquad\qquad M_4\ :=\ M_2'\cup_\Sigma M_1''.
\]
Similarly, we can do this cut-and-glue procedure to $E_j$ to get new vector bundles $E_3$ over $M_3$, $E_4$ over $M_4$. After restricting $D_j,\Phi_j$ to each piece, we obtain Callias-type operators $D_3+\Phi_3$ on $M_3$, $D_4+\Phi_4$ on $M_4$, both having well-defined indexes.

For simplicity, we denote the Callias-type operators $D_j+\Phi_j,\ j=1,2,3,4$ by $P_j$. Then the relative index theorem can be stated as

\begin{theorem}
$\ind P_1+\ind P_2\ =\ \ind P_3+\ind P_4$.
\end{theorem}

Our proof of the theorem involves the gluing formula. However, one has to do deformations to $P_j,\ j=1,2$ first in order to have $(M_j,E_j,P_j)$ along with $\Sigma_j$ satisfy the hypothesis of the gluing formula (cf. Subsections \ref{SS:perturbation} and \ref{SS:pf-relindthm}).

\subsection{Callias-type index theorem}\label{SS:calindthm}

Using the relative index theorem, Anghel proved an important Callias-type index theorem in \cite{Anghel93-2}. Since we give a new proof of the relative index theorem here, we also obtain a new proof of the Callias-type index theorem.

\section{Index of the Operator $\BB_{a,\delta}$}\label{S:indbadelta}

In this section, we construct a family of operators $\BB_{a,\delta}$ on $\tilde{W}:=W\cup\big(\partial{}W\times[0,\infty)\big)$, such that the index $\ind\BB_{a,\delta}=0$. Later in Section \ref{S:pf-thm}, we show that  $\ind\BB_{a,\delta}=\ind(D+\Phi)$ for $a\gg0$.

\subsection{Construction of $\BB_{a,\delta}$}\label{SS:construction}

Consider two anti-commuting actions (``left" and ``right" action) of the Clifford algebra of $\RR$ on the exterior algebra
$\wedge^\bullet\mathbb{C}=\wedge^0\mathbb{C}\oplus\wedge^1\mathbb{C}$ given by
\begin{equation}\label{E:cliffordaction}
c_L(t)\omega=t\wedge\omega-\iota_t\omega,\qquad c_R(t)\omega=t\wedge\omega+\iota_t\omega.
\end{equation}

We define $\tilde W:=W\cup(M\times[0,\infty))$ as in Subsection \ref{SS:outline} and extend the vector bundle $F$ and the operators $\tilde D$, $\tilde\Phi$ to $\tilde W$ in the natural way. Set $\tilde F:=
F\otimes\wedge^\bullet\mathbb{C}$ and consider the operator
\[
B:=\sqrt{-1}(\tilde D+\tilde\Phi)\otimes c_L(1):\,C_0^\infty(\tilde W,\tilde F)\to C_0^\infty(\tilde W,\tilde F).
\]

Let $f:\RR\to[0,\infty)$ be a smooth function with $f(t)=t$ for $t\ge1$, and $f(t)=0$ for $t\le1/2$. Consider the map $p:\tilde W\to\RR$ such that $p(y,t)=f(t)$ for $(y,t)\in M\times(0,\infty)$ and $p(x)=0$ for $x\in W$. For any $a\in\RR$ and $\delta>0$, define the operator
\begin{equation}
	\BB_{a,\delta}\ :=\ B-1\otimes\delta\cdot c_R(p(x)-a).
\end{equation}
Note that as a first order differential operator on the complete manifold $\tilde{W}$, the leading symbol of $\BB_{a,\delta}$ is equal to $\sigma(\tilde D)$. Hence it satisfies \eqref{E:leadingsymbol}. We conclude that  $\BB_{a,\delta}$ is essentially self-adjoint by \cite[Theorem 1.17]{GL}.

\begin{lemma}\label{L:sqfor}
Let $\Pi_i:\tilde F\to F\otimes\wedge^i\mathbb{C}\,(i=0,1)$ be the projections. Then
\begin{equation}\label{E:sqfor}
	\BB_{a,\delta}^2\ =\ (\tilde D+\tilde\Phi)^2\otimes1-\delta\cdot R+\delta^2|p(x)-a|^2,
\end{equation}
where $R$ is a uniformly bounded bundle map whose restriction to $W$ vanishes, and
\begin{equation}\label{E:bundlemapR}
	R|_{M\times(1,\infty)}=\sqrt{-1}\gamma(\Pi_1-\Pi_0), 
	\qquad\text{where}\qquad \gamma|_{F^\pm}=\pm\sqrt{-1}.
\end{equation}
\end{lemma}

\begin{proof}
Note first that $p(x)-a\equiv-a$ on $W$. Thus, since $c_R(-a)$ anti-commutes with $B$, we have $\BB
_{a,\delta}^2|_W=B^2|_W+\delta^2a^2=(\tilde D+\tilde\Phi)^2\otimes1|_W+\delta^2a^2$, which is \eqref{E:sqfor}.

Restricting $\BB_{a,\delta}$ to the cylinder $M\times(0,\infty)$, we obtain
\[
\BB_{a,\delta}|_{M\times(0,\infty)}=\sqrt{-1}(D+\Phi)\otimes c_L(1)+\sqrt{-1}\gamma\otimes c_L(1)
\partial_t-1\otimes\delta(f(t)-a)\cdot c_R(1).
\]
Since $c_L$ and $c_R$ anti-commute, we get
\[
\BB_{a,\delta}^2|_{M\times(0,\infty)}=(\tilde D+\tilde\Phi)^2\otimes1-\sqrt{-1}\delta f'\gamma\otimes c_L(1)
c_R(1)+\delta^2|t-a|^2.
\]
Since $c_L(1)c_R(1)=\Pi_1-\Pi_0$, \eqref{E:sqfor} and \eqref{E:bundlemapR} follow with $R=f'\sqrt{-1}\gamma(\Pi_1-\Pi_0)$.
\end{proof}

\subsection{Fredholmness of $\BB_{a,\delta}$}\label{SS:Fredholmness}

\begin{lemma}\label{L:freopr}
There exists a small enough $\delta$, such that $\BB_{a,\delta}$ is a Fredholm operator for every $a\in\RR$.
\end{lemma}

\begin{proof}
By \cite[Theorem 2.1]{Anghel93}, it is enough to show that the operator $\BB_{a,\delta}$ is invertible at infinity (cf. \eqref{E:invinf}). Since $\BB_{a,\delta}$ is self-adjoint, \eqref{E:invinf} is equivalent to the fact that there exists a constant $\tilde c>0$ and a compact $\tilde K\Subset\tilde W$ such that
\begin{equation}\label{E:invinf2}
(\BB_{a,\delta}^2s,s)\ge\tilde c\,\|s\|^2,\qquad\text{for all }s\in L^2(\tilde W,\tilde F),\;{\rm supp}(s)\cap\tilde K=\emptyset,
\end{equation}
where $(\cdot,\cdot)$ denotes the inner product on $L^2(\tilde W,\tilde F)$. Note that if we denote the bundle map
\[
Q\ :=\ (\tilde\Phi^2+[\tilde D,\tilde\Phi]_+)\otimes1-\delta\cdot R+\delta^2|p(x)-a|^2,
\]
then \eqref{E:sqfor} can also be written as
\[
\BB_{a,\delta}^2\ =\ \tilde D^2\otimes1+Q.
\]
Since $\tilde D^2$ is a non-negative operator on $\tilde W$, \eqref{E:invinf2} can be reduced to
\begin{equation}\label{E:boundofQ}
	|Q(x)|\ge\tilde c,\qquad\text{for all }x\in \tilde W\setminus\tilde K.
\end{equation}

Since both $D+\Phi$ and $\tilde D+\tilde\Phi$ are Callias-type operators, there exist compact subsets $K\Subset M$, $K_W\Subset W$ and positive constants $c,c_W>0$, such that
\begin{equation}\label{E:boundonM}
	|(\Phi^2+[D,\Phi]_+)(y)|\ge c,\qquad\text{for all }y\in M\setminus K,
\end{equation}
and
\begin{equation}\label{E:boundonW}
	|(\tilde\Phi^2+[\tilde D,\tilde\Phi]_+)(x)|\ge c_W,\qquad\text{for all }x\in W\setminus K_W.
\end{equation}

Now consider $|(\tilde\Phi^2+[\tilde D,\tilde\Phi]_+)(y,t)|$ for $(y,t)\in M\times[0,\infty)$. Note that $\tilde\Phi$ is independent of $t$, and anti-commutes with $\gamma\partial_t$. So
\[
[\tilde D,\tilde\Phi]_+=(D+\gamma\partial_t)\tilde\Phi+\tilde\Phi(D+\gamma\partial_t)=D\tilde\Phi+\tilde\Phi D=[D,\tilde\Phi]_+.
\]
Thus 
\[
	|(\tilde\Phi^2+[\tilde D,\tilde\Phi]_+)(y,t)|\ =\ 
	|(\Phi^2+[D,\Phi]_+)(y)|,
\] 
which does not depend on $t$. From \eqref{E:boundonM}, we get
\begin{equation}\label{E:boundoncylinder}
	|(\tilde\Phi^2+[\tilde D,\tilde\Phi]_+)(y,t)|\ge c,\qquad\text{for all }(y,t)\in(M\times[0,\infty))\setminus(K\times[0,\infty)).
\end{equation}
Furthermore, since $K$ is compact, $\tilde\Phi^2+[\tilde D,\tilde\Phi]_+$ is bounded from below on $M\times[0,\infty)$.

Set $W_r:=W\cup(M\times[0,r])$, $\tilde K_r:=K_W\cup(K\times[0,r])$, $r>0$ and $d_1:=\min\{c,c_W\}$. By \eqref{E:boundonW} and \eqref{E:boundoncylinder},
\[
|(\tilde\Phi^2+[\tilde D,\tilde\Phi]_+)(x)|\ge d_1,\qquad\text{for all }x\in W_r\setminus\tilde K_r.
\]
Since $R$ is uniformly bounded on $\tilde W$, we can choose $\delta$ small enough such that 
\[
	\delta\cdot\sup_{x\in\tilde W}|R(x)|\ \le\ d_1/2.
\] 
So
\begin{equation}\label{E:boundofQonWr}
|Q(x)|\ge\frac{d_1}{2},\qquad\text{for all }x\in W_r\setminus\tilde K_r.
\end{equation}
Since $\tilde\Phi^2+[\tilde D,\tilde\Phi]_+$ has a uniform lower bounded on $M\times[0,\infty)$, and $|p(y,t)-a|^2$ grows quadratically as $t\to\infty$, there exist $r=r(a,\delta)$ and $d_2>0$, such that
\begin{equation}\label{E:boundofQoncylinder}
|Q(y,t)|\ge d_2,\qquad\text{for all }(y,t)\in M\times[r,\infty).
\end{equation}

Set $\tilde K:=\tilde K_{r(a,\delta)}$ and $\tilde c:=\min\{d_1/2,d_2\}$. Combining \eqref{E:boundofQonWr} and \eqref{E:boundofQoncylinder} yields \eqref{E:boundofQ}. Therefore the lemma is proved.
\end{proof}

\subsection{Index of $\BB_{a,\delta}$}\label{SS:index} 

From now on we fix  $\delta$ which satisfies Lemma \ref{L:freopr}. Define a grading on the vector bundle $\tilde F=F\otimes\wedge^\bullet\mathbb{C}$ by
\begin{equation}\label{E:tildeF grading}
	\tilde F^+:=F\otimes\wedge^0\mathbb{C}, \qquad 
	\tilde F^-:=F\otimes\wedge^1\mathbb{C},
\end{equation}
\and denote by $\BB_{a,\delta}^\pm:=\BB_{a,\delta}|_{L^2(\tilde W,\tilde F^\pm)}$ the restrictions. We consider the index
\[
\ind\BB_{a,\delta}\ :=\ \dim\ker\BB_{a,\delta}^+-\dim\ker\BB_{a,\delta}^-.
\]

\begin{lemma}\label{L:stbind}
$\ind\BB_{a,\delta}$ is independent of $a$.
\end{lemma}

\begin{proof}
Since for every $a,b\in\RR$ the operator $\BB_{b,\delta}-\BB_{a,\delta}=1\otimes\delta\cdot c_R(b-a)$ is bounded and depends continuously on $b-a\in\RR$, the lemma follows from the stability of the index of a Fredholm operator.
\end{proof}

\begin{lemma}\label{L:ind=0}
${\rm ind}\,\BB_{a,\delta}=0$ for all $a\in\RR$.
\end{lemma}

\begin{proof}
By Lemma \ref{L:stbind}, it suffices to prove this result for a particular value of $a$. If $a$ is a negative number such that $a^2>\sup_{x\in\tilde W}|R(x)|/\delta$, then $\BB_{a,\delta}^2>0$ by \eqref{E:sqfor}, so that $\ker\BB_{a,\delta}=0=\ind\BB_{a,\delta}$.
\end{proof}

\section{The Model Operator}\label{S:model}\label{S:modelopr}

When $a$ is large, all the sections $s\in \ker\BB_{a,\delta}$ are concentrated on the cylinder $M\times[0,\infty)$ near $M\times\{a\}$. Thus index of $\BB_{a,\delta}$ is related to the index of a certain operator on $M\times\RR$, whose restriction to a neighborhood of $M\times\{a\}$ in $\tilde{W}$ is an approximation of the restriction of $\BB_{a,\delta}$ to the neighborhood of $M\times\{a\}$ in $M\times\RR$. We call this operator the {\em model operator} for $\BB_{a,\delta}$ and denote it by $\BB_\delta^{\mo}$. In this section we construct the model operator and show that $\ind{\BB_{\delta}^{\mo}}=\ind(D+\Phi)$. In the next section we show that its index is equal to the index of  $\BB_{a,\delta}$.

\subsection{The operator $\BB_{\delta}^{\mo}$}

Consider the lift of the bundle $E=E^+\oplus E^-$ to the cylinder $M\times\RR$, which will
still be denoted by $E=E^+\oplus E^-$.

Consider the vector bundle $\tilde F^{\mo}=(E^+\oplus E^-)\otimes
\wedge^\bullet\mathbb{C}$  over $M\times\RR$ and the operator
\[
	\BB_\delta^{\mo}:\,
	L^2(M\times\RR,\tilde F^{\mo})\ \to\ L^2(M\times\RR,\tilde F^{\mo})
\]
defined by
\begin{equation}\label{E:Bmod}
	\BB_\delta^{\mo}\ :=\
	\sqrt{-1}(D+\Phi)\otimes c_L(1)+
	\sqrt{-1}\gamma\otimes c_L(1)\partial_t-1\otimes \delta\cdot c_R(t),
\end{equation}
where $t$ is the coordinate along the axis of the cylinder, $\gamma|_{E^\pm}=\pm\sqrt{-1}$, and $\delta$ is fixed with the same value as in Subsection \ref{SS:index}. The operator $\BB_\delta^{\mo}$ satisfies Assumption \ref{A:leadingsymbol} as well and, hence, is self-adjoint. Like in Lemma \ref{L:sqfor}, we have
\[
	(\BB_\delta^{\mo})^2\ =\ (D+\gamma\partial_t+\Phi)^2\otimes1-\sqrt{-1}\delta\gamma(\Pi_1-\Pi_0)+\delta^2t^2.
\]
Then by the same argument as in the proof of Lemma \ref{L:freopr}, $\BB_\delta^{\mo}$ is a Fredholm operator.

Clearly, the restrictions of $\tilde F^{\mo}$ and  $\tilde F$  to the cylinder $M\times(1,\infty)$ are isomorphic. We give $\tilde F^{\mo}$ grading similar to \eqref{E:tildeF grading}, 
\[
	\tilde F_+^{\mo}\ :=\ E\otimes\wedge^0\mathbb{C},\qquad
	\tilde F_-^{\mo}\ :=\ E\otimes\wedge^1\mathbb{C}.
\] 
Set
\[
\ind\BB_\delta^{\mo}\ :=\ \dim\ker(\BB_\delta^{\mo})_+-\dim\ker(\BB_\delta^{\mo})_-,
\]
where $(\BB_\delta^{\mo})_\pm:=\BB_\delta^{\mo}|_{L^2(\tilde W,\tilde F_\pm^{\mo})}$.

\begin{lemma}\label{L:keriso}
The space $\ker\BB_\delta^{\mo}$ is isomorphic (as a graded space) to $\ker(D+\Phi)$. In particular,
\begin{equation}\label{E:keriso}
	\ind\BB_\delta^{\mo}\ =\ \ind(D+\Phi).
\end{equation}
\end{lemma}

\begin{proof}
The space $L^2(M\times\RR,E^\pm\otimes\wedge^\bullet\mathbb{C})$ decomposes into a tensor product 
\[
	L^2(M\times\RR,E^\pm\otimes\wedge^\bullet\mathbb{C})\ = \ 
	L^2(M,E^\pm)\,\otimes\,L^2(\RR,\wedge^\bullet\mathbb{C}).
\]
From \eqref{E:Bmod} it follows that with respect to this decomposition we have
\[
(\BB_\delta^{\mo})^2|_{L^2(M\times\RR,E^\pm\otimes\wedge^\bullet\mathbb{C})}=(D+\Phi)^2\otimes1+
1\otimes(-\partial_{tt}\pm\delta(\Pi_1-\Pi_0)+\delta^2t^2).
\]
Notice that both summands on the right hand side are non-negative.

The space ${\rm ker}\,(-\partial_{tt}+\Pi_1-\Pi_0+t^2)\subset L^2(\RR,\wedge^\bullet\mathbb{C})$ is one-dimensional and is spanned by
\[
	\alpha^+(t)\ =\ e^{-\delta t^2/2}\ \in\ L^2(\RR,\wedge^0\mathbb{C}).
\]
Similarly, the space  ${\rm ker}\,(-\partial_{tt}+\Pi_0-\Pi_1+t^2)$ is one-dimensional and is spanned by \[
	\alpha^-(t)\ =\ e^{-\delta t^2/2}ds\ \in\ L^2(\RR,\wedge^1\mathbb{C}).
\]
It follows that
\[
\ker\,(\BB_\delta^{\mo})^2|_{L^2(M\times\RR,E^\pm\otimes\wedge^\bullet\mathbb{C})}
\simeq\{\sigma\otimes\alpha^\pm(t):\,\sigma\in\ker\,(D+\Phi)^2|_{L^2(M,E^\pm)}\}.
\]
\end{proof}

\subsection{The operator $\BB_{a,\delta}^{\mo}$}

Let
\[
	T_a:\,M\times\RR\to M\times\RR, \qquad T_a(x,t)=(x,t+a)
\]
be the translation and consider the
pull-back map
\[
	T_a^*:\, L^2(M\times\RR,\tilde F^{\mo})\ \to\ L^2(M\times\RR,\tilde F^{\mo}).
\]
Set
\[
\BB_{a,\delta}^{\mo}:=T_{-a}^*\circ\BB_\delta^{\mo}\circ T_a^*=\sqrt{-1}(D+\Phi)\otimes c_L(1)+\sqrt{-1}\gamma\otimes c_L(1)\partial_t-1\otimes\delta\cdot c_R(t-a).
\]
Then

\begin{equation}\label{E:dimkerBmoda=Bmod}
\dim\ker(\BB_{a,\delta}^{\mo})_\pm\ =\ \dim\ker(\BB_\delta^{\mo})_\pm
\end{equation}
for all $a\in\RR$.

\section{Proof of Theorem \ref{T:cobinv2}}\label{S:pf-thm}

In this section, we finish the proof of the cobordism invariance of the Callias-type index by showing that $\ind\BB_{a,\delta}=
\ind\BB_\delta^{\mo}$. Since $\delta$ is fixed throughout the section, we omit it  from the notation, and write $\BB_a,\BB^{\mo}$ and $\BB_a^{\mo}$ for $\BB_{a,\delta},
\BB_\delta^{\mo}$ and $\BB_{a,\delta}^{\mo}$, respectively.

\subsection{The spectral counting function}\label{SS:spectral counting function}

For a self-adjoint operator $P$ and a real number  $\lambda$, we denote by $N(\lambda,P)$ the number of eigenvalues of $P$ not exceeding $\lambda$ (counting with multiplicities). If the intersection of the continuum spectrum of $P$ with the set $(-\infty,\lambda]$ is not empty, then  we set $N(\lambda,P)=\infty$.

Let $\BB_a^\pm$ denote the restrictions of $\BB_a$ to the spaces $L^2(\tilde W,\tilde F^\pm)$ and let $\BB_\pm^{\mo}$, $\BB_{a,\pm}^{\mo}$ denote the restrictions of $\BB^{\mo}$, $\BB_a^{\mo}$ to the spaces $L^2(M\times\RR,\tilde F^{\mo}_\pm)$.

Since the operator $\BB^{\mo}$ is self-adjoint, by von Neumann's theorem (cf. \cite[Theorem \upper{\romannumeral10}.25]{ReSi2}), the operators $(\BB^{\mo})^2_\pm=\BB_\mp^{\mo}\BB_\pm^{\mo}=(\BB_\pm^{\mo})^*\BB_\pm^{\mo}$ are also self-adjoint. Since the operators  $(\BB^{\mo})_\pm^2$ are Fredholm,  they have smallest non-zero elements of the spectra, denoted by $\lambda_\pm$.

\begin{lemma}\label{L:lambda+=lambda-}
{$\lambda_+=\lambda_-$.}
\end{lemma}

\begin{proof}
Since $(\BB^{\mo})^2_+=\BB_-^{\mo}\BB_+^{\mo}$, $(\BB^{\mo})^2_-=\BB_+^{\mo}\BB_-^{\mo}$, by \cite[Theorem 1.1]{Hardt00}, their spectra satisfy
\[
	\sigma((\BB^{\mo})^2_+)\setminus\{0\}\ =\ \sigma((\BB^{\mo})^2_-)\setminus\{0\}.
\]
In particular, $\lambda_+=\lambda_-$.
\end{proof}

From now on, we set 
\[
	\lambda\ :=\ \lambda_+\ = \ \lambda_-.
\]

\begin{proposition}\label{P:n=dker}
For any $0<\epsilon<\lambda$, there exists $A=A(\epsilon,\delta,p)>0$, such that
\begin{equation}\label{E:n=dker}
N(\lambda-\epsilon,(\BB_a^2)^\pm)\ =\ \dim\ker(\BB^{\mo})_\pm^2,\qquad
\mbox{for all $a>A$},
\end{equation}
where $\delta>0$, $p:\tilde{W}\to \RR$ are as in Subsection \ref{SS:construction}. In particular, 
\begin{equation}\label{E:N=ind}
	\ind \BB^{\mo} \ = \ N(\lambda-\epsilon,(\BB_a^2)^+)
	\ - \ N(\lambda-\epsilon,(\BB_a^2)^-).
\end{equation}
\end{proposition}

Before proving this proposition we show how it implies Theorem \ref{T:cobinv2}.

\subsection{Proof of Theorem \ref{T:cobinv2}}\label{SS:pf-cobinv2}

By Proposition \ref{P:n=dker}, $N(\lambda-\epsilon,(\BB_a^2)^\pm)<\infty$. Let \
\[
	V_{\epsilon,a}^\pm\ \subset\ L^2(\tilde W,\tilde F^\pm)
\] 
denote the vector spaces spanned by the eigenvectors of the operators $(\BB_a^2)^\pm$ with eigenvalues within $(0,\lambda-\epsilon]$. Then $\dim{}V_{\epsilon,a}^\pm<\infty$ and the restrictions of the operators $\BB_a^\pm$ to $V_{\epsilon,a}^\pm$ define bijections
\[
	\BB_a^\pm:\,V_{\epsilon,a}^\pm \ \overset{~}\longrightarrow\
		V_{\epsilon,a}^\mp.
\]

Hence, 
\[
	\dim V_{\epsilon,a}^+\ =\ \dim V_{\epsilon,a}^-.
\]
Thus
\begin{eqnarray*}
	N(\lambda-\epsilon,(\BB_a^2)^+)-N(\lambda-\epsilon,(\BB_a^2)^-)
	\!\!&=&\!\!
	(\dim\ker\BB_a^++\dim V_{\epsilon,a}^+)-(\dim\ker\BB_a^-+\dim V_{\epsilon,a}^-)\\
	\!\!&=&\!\! \dim\ker\BB_a^+-\dim\ker\BB_a^-\ = \ 
	\ind \BB_a.
\end{eqnarray*}
From Proposition \ref{P:n=dker} we now obtain
\[
	\ind\BB_a\ =\ \ind\BB^{\mo}
\] 
and  Theorem \ref{T:cobinv2} follows from Lemma \ref{L:ind=0} and \ref{L:keriso}.
\hfill$\square$

\bigskip
The rest of this section is occupied with the proof of Proposition \ref{P:n=dker}.

\subsection{Estimate from above on $N(\lambda-\epsilon,(\BB_a^2)^\pm)$}\label{SS:estabv}

First we show that
\begin{equation}\label{E:estabv}
N(\lambda-\epsilon,(\BB_a^2)^\pm)\ \le\ \dim\ker\BB_\pm^{\mo}.
\end{equation}
This is done through the following techniques.

\subsection{The IMS localization}\label{SS:ims}

Let $j,\hat{j}:\RR\to[0,1]$ be smooth functions such that $j^2+\hat{j}^2\equiv1$ and $j(t)=1$ for $t\ge3$, while $j(t)=0$ for $t\le2$. Set $j_a(t)=j(a^{-1/2}t),\,\hat{j}_a(t)=\hat{j}(a^{-1/2}t)$. Now we view them as functions on the cylinder $M\times[0,\infty)$ (whose points are written as $(y,t)$). Similarly, we still use the same notations $j_a(x)=j(a^{-1/2}p(x)),\,\hat{j}_a(x)=\hat{j}
(a^{-1/2}p(x))$ to denote the functions on $\tilde W$, where $p(x)$ is defined in Subsection \ref{SS:construction} .

We use the following verison of the IMS localization, cf. \cite[\S3]{Shubin96Morse}, \cite[Lemma~4.5]{Br-cobinv}%
\footnote{The abbreviation IMS is formed by the initials of the surnames of R.~Ismagilov, J.~Morgan, I.~Sigal and B.~Simon.}

\begin{lemma}\label{L:ims}
The following operator identity holds:
\begin{equation}\label{E:ims}
	\BB_a^2\ =\ \hat{j}_a\BB_a^2\hat{j}_a+j_a\BB_a^2j_a+\frac{1}{2}[\hat{j}_a,[\hat{j_a},\mathbf
{B}_a^2]]+\frac{1}{2}[j_a,[j_a,\BB_a^2]].
\end{equation}
\end{lemma}

Now we estimate each summand on the right-hand side of \eqref{E:ims}.

\begin{lemma}\label{L:estabv-1}
There exists $A=A(\delta,p)>0$ such that 
\[
	\hat{j}_a\BB_a^2\hat{j}_a\ \ge\ 
	\frac{\delta^2a^2}{8}\hat{j}_a^2 
\]
for all $a>A$.
\end{lemma}

\begin{proof}
Note that if $x\in{\rm supp}\,\hat{j}_a$, then $p(x)\le3a^{1/2}$. Hence for $a>36$, we have 
\[
	\hat{j}_a^2|p(x)-a|^2\ \ge\ \frac{a^2}{4}\hat{j}_a^2.
\]

Set $A=\max\{36,4\delta^{1/2}\sup_{x\in\tilde W}\|R(x)\|^{1/2}\}$ and let $a>A$. By Lemma \ref{L:sqfor},
\[
\hat{j}_a\BB_a^2\hat{j}_a\ge\hat{j}_a^2\delta^2|p(x)-a|^2-\hat{j}_a(\delta\cdot R)\hat{j}_a\ge\frac
{\delta^2a^2}{8}\hat{j}_a^2.
\]
\end{proof}

\subsection{\empty}\label{SS:estabv2}

Let $\Pi_a:L^2(M\times\RR,\tilde F^{\mo})\to\ker\BB_a^{\mo}$ be the orthogonal projection and $\Pi_a^\pm$ be the restrictions of $\Pi_a$ to the spaces $L^2(M\times\RR,
\tilde F_\pm^{\mo})$. Then $\Pi_a^\pm$ are finite rank operators and their ranks are $\dim\ker\BB_{a,\pm}^{\mo}$, which are equal to $\dim\ker\BB_\pm^{\mo}$ by \eqref{E:dimkerBmoda=Bmod}. Since $(\BB_a^{\mo})_\pm^2$ are nonnegative operators, it's clear that
\begin{equation}\label{E:estfor}
	(\BB_a^{\mo})_\pm^2+\lambda\Pi_a^\pm\ \ge\ \lambda.
\end{equation}
Observe that ${\rm supp}\,j_a$ in $M\times\RR$ is a subset of $M\times[0,\infty)$. It's a subset of $\tilde W=W\cup(M\times[0,\infty))$ as well. So we can consider $j_a\Pi_aj_a$ and $j_a\BB_a^{\mo}j_a$ as operators on $\tilde W$. Then $j_a\BB_a^2j_a=j_a
(\BB_a^{\mo})^2j_a$. Hence, \eqref{E:estfor} implies the following.

\begin{lemma}\label{L:estabv-2}
$j_a(\BB_a^2)^\pm j_a+\lambda j_a\Pi_a^\pm j_a\ge\lambda j_a^2,\quad {\rm rank}\,j_a\Pi_a^\pm j_a\le
\dim\ker\BB_\pm^{\mo}$.
\end{lemma}

The next lemma estimates the last two summands on the right-hand side of \eqref{E:ims}.

\begin{lemma}\label{L:estabv-3}
Let $C=2\max\big\{\max\{|j'(t)|^2,|\hat{j}'(t)|^2\}:t\in\RR\big\}$. Then
\begin{equation}\label{E:estabv-3}
	\|[j_a,[j_a,\BB_a^2]]\|\le Ca^{-1},\quad \|[\hat{j}_a,[\hat{j}_a,\BB_a^2]]\|\le Ca^{-1}\qquad\mbox{for all $ a>0$}.
\end{equation}
\end{lemma}

\begin{proof}
By Lemma \ref{L:sqfor}, we get
\[
\|[j_a,[j_a,\BB_a^2]]\|=2|j'_a(t)|^2=2a^{-1}|j'(a^{-1/2}t)|^2,
\]
\[
\|[\hat{j}_a,[\hat{j}_a,\BB_a^2]]\|=2|\hat{j}'_a(t)|^2=2a^{-1}|\hat{j}'(a^{-1/2}t)|^2.
\]
Then \eqref{E:estabv-3} follows immediately.
\end{proof}

Since $\lambda$ is fixed, combining Lemma \ref{L:ims}, \ref{L:estabv-1}, \ref{L:estabv-2} and \ref{L:estabv-3}, we obtain

\begin{corollary}\label{C:estabv-all}
For any $\epsilon>0$, there exists $A=A(\epsilon,\delta,p)>0$ such that, for all $a>A$,
\begin{equation}\label{E:estabv-all}
(\BB_a^2)^\pm+\lambda j_a\Pi_a^\pm j_a\ge\lambda-\epsilon,\qquad {\rm rank}\,j_a\Pi_a^\pm j_a\le\dim\ker\BB_\pm^{\mo}.
\end{equation}
\end{corollary}

The estimate \eqref{E:estabv} now follows from Corollary \ref{C:estabv-all} and the following result, \cite[p.~270]{ReSi4}:

\begin{lemma}\label{L:multevalue}
Assume that $P,Q$ are self-adjoint operators on a Hilbert space such that ${\rm rank}\,Q\le k$ and there exists $\mu>0$ such that $\langle(P+Q)u,u\rangle\ge\mu\langle u,u\rangle$ for any $u\in{\rm Dom}(P)$. Then $N(\mu-\epsilon,P)\le k$ for any $\epsilon>0$.
\end{lemma}

\subsection{Estimate from below on $N(\lambda-\epsilon,(\BB_a^2)^\pm)$}\label{SS:estblw}

Now it remains to prove that
\begin{equation}\label{E:estblw}
	N(\lambda-\epsilon,(\BB_a^2)^\pm)\ \ge\ \dim\ker\BB_\pm^{\mo}\ =\ \dim\ker\BB_{a,\pm}^{\mo}.
\end{equation}

By \eqref{E:estabv}, $N(\lambda-\epsilon,(\BB_a^2)^\pm)$ are finite for $a$ large enough. Under this circumstance, let $\tilde V_{\epsilon,a}^\pm\subset L^2(\tilde W,\tilde F^\pm)$ denote the vector spaces spanned by the eigenvectors of the operators $(\BB_a^2)^\pm$ for eigenvalues within $[0,\lambda-\epsilon]$. Let $\Theta_{\epsilon,a}^\pm:L^2(\tilde W,\tilde F^\pm)\to\tilde V_{\epsilon,a}^\pm$ be the orthogonal projections. Then ${\rm rank}\,\Theta_{\epsilon,a}^\pm=N(\lambda-\epsilon,(\BB_a^2)^\pm)$. As in Subsection \ref{SS:estabv2}, we can consider $j_a\Theta_{\epsilon,a}^\pm j_a$ as operators on $L^2(M\times\RR,\tilde F_\pm^{\mo})$. Then the same argument as in the proof of Corollary \ref{C:estabv-all} works here and we have
\begin{lemma}\label{L:estblw-all}
For any $\epsilon>0$, there exists $A=A(\epsilon,\delta)>0$ such that, for all $a>A$,
\begin{equation}\label{E:estblw-all}
(\BB_a^{\mo})_\pm^2+\lambda j_a\Theta_a^\pm j_a\ge\lambda-\epsilon,\quad {\rm rank}\,j_a
\Theta_a^\pm j_a\le\dim N(\lambda-\epsilon,(\BB_a^2)^\pm).
\end{equation}
\end{lemma}

Similarly, the estimate \eqref{E:estblw} follows from Lemma \ref{L:estblw-all} and \ref{L:multevalue}.

Now the proof of Proposition \ref{P:n=dker} and, hence, of Theorem \ref{T:cobinv2} is complete. \hfill$\square$

\section{The Gluing Formula}\label{S:gluing}

Our first application of Theorem~\ref{T:cobinv} is the gluing formula. If we cut a complete manifold along a hypersurface $\Sigma$, we obtain a manifold with boundary. By rescaling the metric near the boundary, we may convert it to a complete manifold without boundary. In this section, we show that the index of a Callias-type operator is invariant under this type of surgery. In particular, if $M$ is partitioned into two pieces $M_1$ and $M_2$ by $\Sigma$, we see that the index on $M$ is equal to the sum of the indexes on $M_1$ and $M_2$. In other words, the index is additive.

\subsection{The surgery}\label{SS:surgery}

Let $(M,E,D+\Phi)$ be as in Subsection \ref{SS:indexofCallias} with $\dim M=n$, where 
\[ 
	D:\,C_0^\infty(M,E^\pm)\ \to\ C_0^\infty(M,E^\mp)
\] 
satisfies Assumption \ref{A:leadingsymbol} and $D+\Phi$ is a Callias-type operator. Suppose $\Sigma\subset M$ is a smooth hypersurface. For simplicity, we assume that $\Sigma$ is compact. 

Throughout this section we make the following assumption.

\begin{assumption}\label{A:gluing}
There exist a compact set $K\Subset M$ and two constants $c_1,c_2>0$ such that
\begin{enumerate}
\item for all $x\in M\setminus K$,
\[
|(\Phi^2+[D,\Phi]_+)(x)|\ge c_1,\qquad|\Phi^2(x)|\ge c_2;
\]
\item $\Sigma\subset M\setminus K$, which indicates that $K$ is still a compact subset of $M_\Sigma:=M\setminus\Sigma$.
\end{enumerate}
\end{assumption}

We denote by $E_\Sigma$ the restriction of the graded vector bundle $E$ to $M_\Sigma$. Let $g$ denote the Riemannian metric on $M$. By a rescaling of $g$ near $\Sigma$, one can obtain a complete Riemannian metric on $M_\Sigma$ and a Callias-type operator $D_\Sigma+\Phi_\Sigma$ on $M_\Sigma$. It follows from the cobordism invariance of the index (cf. Theorem~\ref{T:cobinv}) that the index of $D_\Sigma+\Phi_\Sigma$ is independent of the choice of a rescaling. 

\subsection{A rescaling of the metric}\label{SS:rescaling}
We now present one of the possible constructions of a complete metric on $M_\Sigma$. 

Let $\tau:M\to[-1,1]$ be a smooth function, such that $\tau^{-1}(0)=\Sigma$ and $\tau$ is regular at $\Sigma$. Set $\alpha(x)=(\tau(x))^2$. Define the metric $g_{_\Sigma}$ on $M_\Sigma$ by
\begin{equation}\label{E:defofgmsigma}
	g_{_\Sigma}\ :=\ \frac{1}{\alpha(x)^2}g.
\end{equation}
This makes $(M_\Sigma,g_{_\Sigma})$ a complete Riemannian manifold. 

Let $d{\rm vol}_g(x)$ and $d{\rm vol}_{g_{_\Sigma}}(x)$ denote the canonical volume forms on $(M,g)$ and $(M_\Sigma,g_{_\Sigma})$, respectively. It's easy to see that $d{\rm vol}_{g_{_\Sigma}}(x)=\frac1{\alpha(x)^n}d{\rm vol}_g(x)$. So the $L^2$-inner product on $L^2(M_\Sigma,E_\Sigma)$ becomes
\begin{equation}\label{E:L2sigmainnerproduct}
	(s_1,s_2)_{\Sigma}\ =\ \int_{M_\Sigma}\langle s_1(x),s_2(x)\rangle_{(E_\Sigma)_x}\frac{1}{\alpha(x)^n}d{\rm vol}_g(x).
\end{equation}

\subsection{The Callias-type operator on $(M_\Sigma,g_{_\Sigma})$}\label{SS:CalliasonMsigma}

In order to get a natural Callias-type operator acting on $C_0^\infty(M_\Sigma,E_\Sigma)$ we set
\[
\Phi_\Sigma\ :=\ \Phi|_{M_\Sigma},
\]
and
\begin{equation}\label{E:Dsigma}
	D_\Sigma(x)(s)\ :=\ 
	\alpha(x)^{\frac{n+1}{2}}D(x)(\alpha(x)^{-\frac{n-1}{2}}s),
	\qquad\text{for all }x\in M_\Sigma,\,s\in C_0^\infty(M_\Sigma,E_\Sigma).
\end{equation}
It's easy to check that 
\[
	\sigma(D_\Sigma)(x,\xi)\ =\ \alpha(x)\sigma(D)(x,\xi)
\]
So $D_\Sigma$ also satisfies Assumption \ref{A:leadingsymbol}. Thus $D_\Sigma$ and $D_\Sigma+\Phi_\Sigma:C_0^\infty(M_\Sigma,E_\Sigma^\pm)\to C_0^\infty(M_\Sigma,E_\Sigma^\mp)$ are still $\ZZ_2$-graded first-order elliptic operators, which are essentially self-adjoint with respect to the $L^2$-inner product defined by \eqref{E:L2sigmainnerproduct}.

\begin{remark}\label{R:Diracoperator-sigma}
If $E$ is a Clifford bundle with respect to $g$, and $D$ is the Dirac operator, then $E_\Sigma$ also has a Clifford structure with respect to $g_{_\Sigma}$, and $D_\Sigma$ defined by \eqref{E:Dsigma} is precisely the associated Dirac operator.
\end{remark}

\begin{lemma}\label{L:Callias-sigma}
$D_\Sigma+\Phi_\Sigma$ is a Callias-type operator, and, hence, is Fredholm.
\end{lemma}

\begin{proof}
Since $[D,\Phi]_+$ is a bundle map, a direct computation gives that
\begin{eqnarray*}
[D_\Sigma,\Phi_\Sigma]_+(s)\!\!&=&\!\!D_\Sigma\Phi_\Sigma(s)+\Phi_\Sigma D_\Sigma(s)\\
\!\!&=&\!\!\alpha^{\frac{n+1}{2}}D(\alpha^{-\frac{n-1}{2}}\Phi(s))+\Phi(\alpha^{\frac{n+1}{2}}
D(\alpha^{-\frac{n-1}{2}}s))\\
\!\!&=&\!\!\alpha^{\frac{n+1}{2}}D(\Phi(\alpha^{-\frac{n-1}{2}}s))+\alpha^{\frac{n+1}{2}}\Phi
(D(\alpha^{-\frac{n-1}{2}}s))\\
\!\!&=&\!\!\alpha^{\frac{n+1}{2}}[D,\Phi]_+(\alpha^{-\frac{n-1}{2}}s)\,=\,\alpha[D,\Phi]_+({s}).
\end{eqnarray*}
So $[D_\Sigma,\Phi_\Sigma]_+$ is a bundle map as well. Then
\[
\Phi_\Sigma^2+[D_\Sigma,\Phi_\Sigma]_+=(\Phi^2+\alpha[D,\Phi]_+)|_{M_\Sigma}=((1-\alpha)\Phi^2+\alpha
(\Phi^2+[D,\Phi]_+))|_{M_\Sigma}.
\]
Note that $\alpha(x)\in[0,1]$, by Assumption \ref{A:gluing},
\[
|(\Phi_\Sigma^2+[D_\Sigma,\Phi_\Sigma]_+)(x)|\ge c,\qquad\text{for all }x\in M_\Sigma\setminus K,
\]
where $c:=\min\{c_1,c_2\}$. Thus $D_\Sigma+\Phi_\Sigma$ is a Callias-type operator and, hence, is Fredholm by Lemma \ref{L:Callias-invertible}.
\end{proof}

It follows from the above lemma that the index $\ind(D_\Sigma+\Phi_\Sigma)$ is well defined.

\subsection{The gluing formula}\label{SS:glufor}

Under the above setting, there are two well-defined indexes $\ind(D+\Phi)$ and $\ind(D_\Sigma+\Phi_\Sigma)$. 

\begin{theorem}\label{T:glufor}
The operators $D+\Phi$ and $D_\Sigma+\Phi_\Sigma$ are cobordant in the sense of Definition~\ref{D:calliascob}. In particular,
\[
	\ind(D+\Phi)\ =\ \ind(D_\Sigma+\Phi_\Sigma).
\]
\end{theorem}

We refer to Theorem~\ref{T:glufor} as a \emph{gluing formula}, meaning that $M$ is obtained from $M_\Sigma$ by gluing along $\Sigma$.

\begin{proof}
The goal is to find a triple $(W,F,\tilde D+\tilde\Phi)$, such that it is the cobordism between $(M,E,D+\Phi)$ and $(M_\Sigma,E_\Sigma,D_\Sigma+\Phi_\Sigma)$ and then apply Theorem \ref{T:cobinv}.

Consider
\[
W\ :=\ \Big\{(x,t)\in M\times[0,\infty):t\le\frac{1}{\alpha(x)}+1\Big\}.
\]
Then $W$ is a noncompact manifold whose boundary is diffeomorphic to the disjoint union of $M\simeq M\times\{0\}$ and $M\setminus\Sigma\simeq\{(x,\frac1{\alpha(x)}+1)\}$. Essentially, $W$ is the required cobordism. However, to be precise, we need to define a complete Riemannian metric $g^W$ on $W$, such that condition (\romannumeral1) of Definition~\ref{D:calliascob} is fulfilled.

Let $\beta:W\to[0,1]$ be a smooth function such that $\beta(x,t)=1$ for $0\le t\le1/2$, $\beta(x,t)>0$ for $1/2<t<1/\alpha(x)+1/2$ and $\beta(x,t)=\alpha(x)$ for $1/\alpha(x)+1/2\le t\le1/\alpha(x)+1$.
Define the metric $g^W$ on $W$ by
\[
g^W((\xi_1,\eta_1),(\xi_2,\eta_2))\ :=\ \frac{1}{\beta(x,t)^2}g(\xi_1,\xi_2)+\eta_1\eta_2,
\]
where $(\xi_1,\eta_1),(\xi_2,\eta_2)\in T_xM\oplus\RR\simeq T_{(x,t)}W$. Then $g^W$ is a complete metric.

Consider the neighborhood
\[
U\ :=\ \{(x,t):0\le t<1/3\}\sqcup\Big\{(x,t):\frac{1}{\alpha(x)}+\frac{2}{3}<t\le\frac{1}{\alpha(x)}+1\Big\}
\]
of $\partial W$. Then define a map $\phi:(M\times[0,1/3))\sqcup(M_\Sigma\times(-1/3,0])\to U$ by the formulas
\begin{eqnarray*}
\phi(x,t)\!\!&:=&\!\!(x,t),\qquad x\in M,\,0\le t<1/3,\\
\phi(x,t)\!\!&:=&\!\!\Big(x,\frac{1}{\alpha(x)}+1+t\Big),\qquad x\in M_\Sigma,\,-1/3<t\le0.
\end{eqnarray*}
It's easy to see that $\phi$ is a metric-preserving diffeomorphism, satisfying condition (\romannumeral1) of Definition~\ref{D:calliascob}.

Let $\pi:M\times[0,\infty)\to M$ be the projection. Then the pull-back $\pi^* E$ is a vector bundle over $M\times[0,\infty)$. Define
\[
F\ :=\ \pi^* E|_W.
\]
So $F$ is a vector bundle over $W$, whose restriction to the first part of $U$ is isomorphic to the lift of $E$ over $M\times[0,1/3)$ and whose restriction to the second part of $U$ is isomorphic to the lift of $E_\Sigma$ over $M_\Sigma\times(-1/3,0]$. Hence condition (\romannumeral2) of Definition~\ref{D:calliascob} is fulfilled. Note that here we can give $F$ a natural grading which is compatible with that on $E$ and $E_\Sigma$:
\[
F^+\ :\ =\pi^* E^+|_W,\qquad F^-\ :=\ \pi^* E^-|_W.
\]

We still use $D$ and $\Phi$ to denote the lifts of $D$ and $\Phi$ to $M\times[0,\infty)$. Now we define
\[
	\tilde D,\tilde\Phi:\,C_0^\infty(W,F)\ \to\ C_0^\infty(W,F)
\] 
by
\begin{equation}\label{E:tildeD}
  \begin{aligned}
	\tilde D(\tilde s)\ :=\ (\beta^{\frac{n+1}{2}}D|_W)(\beta^{-\frac{n-1}{2}}\tilde s)&+\gamma\partial_t(\tilde s),\qquad\text{for all }\tilde s\in C_0^\infty(W,F),\\
\tilde\Phi\ &:=\ \Phi|_W,
  \end{aligned}
\end{equation}
where $\gamma|_{F^\pm}=\pm\sqrt{-1}$. Then $\sigma(\tilde D)=\beta(\sigma(D|_W))+\sigma(\gamma\partial_t)$. Since $\beta$ lies in $[0,1]$, $\tilde D$ satisfies Assumption \ref{A:leadingsymbol}. Moreover, $\tilde D+\tilde\Phi$ takes the form $D+\gamma\partial_t+\Phi$ on one end $M\times[0,1/3)$ and the form $D_\Sigma+\gamma\partial_t+\Phi_\Sigma$ on the other end $M_\Sigma\times(-1/3,0]$. So $\tilde D+\tilde\Phi$ has exactly the form required in condition (\romannumeral3) of Definition \ref{D:calliascob}.

It remains to verify that $\tilde D+\tilde\Phi$ is a Callias-type operator. Note that $\gamma\partial_t$ anti-commutes with $\tilde\Phi$, by the same computation as in the proof of Lemma \ref{L:Callias-sigma}, we have
\[
[\tilde D,\tilde\Phi]_+\ =\ \beta[D,\Phi]_+|_W
\]
is a bundle map. And
\[
\tilde\Phi^2+[\tilde D,\tilde\Phi]_+\ =\ ((1-\beta)\Phi^2+\beta(\Phi^2+[D,\Phi]_+))|_W.
\]
By Assumption \ref{A:gluing}, 
\[
\tilde K\ :=\ \Big\{(x,t)\in M\times[0,\infty):x\in K,\,t\le\frac{1}{\alpha(x)}+1\Big\}
\]
is a compact subset of $W$, and
\[
|(\Phi^2+[D,\Phi]_+)(x,t)|\ge c_1,\quad|\Phi^2(x,t)|\ge c_2,\qquad\text{for all }(x,t)\in W\setminus\tilde K.
\]
Again by $\beta\subset[0,1]$, we get
\[
|(\tilde\Phi^2+[\tilde D,\tilde\Phi]_+)(x,t)|\ge c,\qquad\text{for all }(x,t)\in W\setminus\tilde K,
\]
where $c$ is the same as in the proof of Lemma \ref{L:Callias-sigma}. Thus $\tilde D+\tilde\Phi$ is a Callias-type operator, and condition (\romannumeral3) of Definition \ref{D:calliascob} is also fulfilled.

Therefore, $(W,F,\tilde D+\tilde\Phi)$ is a cobordism between $(M,E,D)$ and $(M_\Sigma,E_\Sigma,D_\Sigma+\Phi_\Sigma)$, and by Theorem \ref{T:cobinv}, $\ind(D+\Phi)=\ind(D_\Sigma+\Phi_\Sigma)$.
\end{proof}

\subsection{The additivity of the index}\label{SS:additivity}

Suppose that $M$ is partitioned into two relatively open submanifolds $M_1$ and $M_2$ by $\Sigma$, so that $M_\Sigma=M_1\sqcup M_2$. The metric $g_{_\Sigma}$ induces complete Riemannian metrics $g_{_{M_1}},g_{_{M_2}}$ on $M_1$ and $M_2$, respectively. Let $E_i,D_i,\Phi_i$ denote the restrictions of the graded vector bundle $E_\Sigma$ and operators $D_\Sigma,\Phi_\Sigma$ to $M_i$ ($i=1,2$). Then Theorem \ref{T:glufor} implies the following corollary.

\begin{corollary}\label{C:additivity}
$\ind(D+\Phi)=\ind(D_1+\Phi_1)+\ind(D_2+\Phi_2)$.
\end{corollary}

Thus we see that the index is ``additive".

\section{Relative Index Theorem for Callias-type Operators}\label{S:relindthm}

As a second application of Theorem \ref{T:cobinv}, and also as an application of Corollary \ref{C:additivity}, we give a new proof of the relative index theorem for Callias-type operators. There are several different forms of relative index theorem. In this paper we follow the approach of \cite{Bunke}.

\subsection{Setting}\label{SS:setting}

Let $(M_j,E_j,D_j+\Phi_j),\ j=1,2$ be two triples of complete Riemannian manifold endowed with a $\ZZ_2$-graded Hermitian vector bundle and with the associated Callias-type operator acting on the compactly supported smooth sections of the bundle. Suppose they satisfy Assumption \ref{A:gluing}.(\romannumeral1) of Subsection~\ref{SS:surgery}. In particular, the indexes $\ind(D_j+\Phi_j),\ j=1,2$ are well-defined.

Suppose $M_j'\cup_{\Sigma_j}M_j''$ are partitions of $M_j$ into relatively open submanifolds, where $\Sigma_j$ are compact hypersurfaces. We make the following assumption.
\begin{assumption}\label{A:relindthm}
There exist tubular neighborhoods $U(\Sigma_1)$, $U(\Sigma_2)$ of $\Sigma_1$ and $\Sigma_2$ such that:
\begin{enumerate}
\item there is a commutative diagram of isometric diffeomorphisms 
\[
\begin{array}{ccccc}
\psi & : & E_1|_{U(\Sigma_1)} & \rightarrow & E_2|_{U(\Sigma_2)} \\
 & & \downarrow & & \downarrow \\
\phi & : & U(\Sigma_1) & \rightarrow & U(\Sigma_2) \\
 & & \uparrow & & \uparrow \\
\phi|_{\Sigma_1} & : & \Sigma_1 & \rightarrow & \Sigma_2
\end{array}
\]
\item $\Phi_j$ are invertible bundle maps on $U(\Sigma_j)$, $j=1,2$.
\item $D_1$ and $D_2$, $\Phi_1$ and $\Phi_2$ coincide on the neighborhoods, i.e., 
\[
\psi\circ D_1\ =\ D_2\circ\psi,\qquad\qquad\psi\circ\Phi_1\ =\ \Phi_2\circ\psi.
\]
\end{enumerate}
\end{assumption}

We cut  $M_j$ along $\Sigma_j$ and use the map $\phi$ to glue the pieces together interchanging $M_1''$ and $M_2''$. In this way we obtain the manifolds 
\[
M_3\ :=\ M_1'\cup_\Sigma M_2'',\qquad\qquad M_4\ :=\ M_2'\cup_\Sigma M_1'',
\]
where $\Sigma\cong\Sigma_1\cong\Sigma_2$. We use the map $\psi$ to cut the bundles $E_1$, $E_2$ at $\Sigma_1$, $\Sigma_2$ and glue the pieces together interchanging $E_1|_{M_1''}$ and $E_2|_{M_2''}$. With this procedure we obtain $\ZZ_2$-graded Hermitian vector bundles $E_3\rightarrow M_3$ and $E_4\rightarrow M_4$. At last, we define $D_3$ and $D_4$, $\Phi_3$ and $\Phi_4$ by
\[
D_3\ =\ \left\{
\begin{array}{l}
D_1\quad\text{on }M_1' \\
D_2\quad\text{on }M_2''
\end{array}
\right.,\qquad\qquad
D_4\ =\ \left\{
\begin{array}{l}
D_2\quad\text{on }M_2' \\
D_1\quad\text{on }M_1''
\end{array}
\right.;
\]
\[
\Phi_3\ =\ \left\{
\begin{array}{l}
\Phi_1\quad\text{on }M_1' \\
\Phi_2\quad\text{on }M_2''
\end{array}
\right.,\qquad\qquad
\Phi_4\ =\ \left\{
\begin{array}{l}
\Phi_2\quad\text{on }M_2' \\
\Phi_1\quad\text{on }M_1''
\end{array}
\right..
\]
Then by Assumption \ref{A:relindthm}.(\romannumeral3), $D_j+\Phi_j:C_0^\infty(M_j,E_j)\to C_0^\infty(M_j,E_j),\ j=3,4$ are also $\ZZ_2$-graded essentially self-adjoint Callias-type operators. So again we have two well-defined indexes $\ind(D_3+\Phi_3)$ and $\ind(D_4+\Phi_4)$.

\subsection{Relative index theorem}\label{SS:relindthm}

As in Subsection \ref{SS:relindthm-intro}, we define $P_j:=D_j+\Phi_j,\ j=1,2,3,4$. The we have the following version of the relative index theorem 

\begin{theorem}\label{T:relindthm}
$\ind P_1+\ind P_2\ =\ \ind P_3+\ind P_4$.
\end{theorem}

The idea of the proof is to use Corollary \ref{C:additivity} to write $\ind P_j$ as the sum of the indexes on two pieces.  However, as one might notice, in our setting, $\Sigma_1$ and $\Sigma_2$ might not satisfy condition (\romannumeral2) of Assumption~\ref{A:gluing}. So Corollary~\ref{C:additivity} cannot be applied directly. In the next subsection, we construct deformations of the operators $P_1$ and $P_2$ which preserve the indexes such that the deformed operators satisfy Assumption \ref{A:gluing}.(\romannumeral2).

\subsection{Deformations of the operators $P_1$ and $P_2$}\label{SS:perturbation}

Let $U_j,\ j=1,2$ denote the neighborhoods $U(\Sigma_j)$ of $\Sigma_j$ in Subsection \ref{SS:setting}. Since $\Sigma_j$ are compact hypersurfaces, we can find their relatively compact neighborhoods $V_j,W_j$ satisfying $V_2=\phi(V_1),W_2=\phi(W_1)$ and
\[
V_j\subset\overline{V_j}\subset W_j\subset\overline{W_j}\subset U_j.
\]

Fix smooth functions $f_j:M_j\to[0,1]$ such that  $f_j\equiv1$ on $\overline{V_j}$ and $f_j\equiv0$ outside of $W_j$. Notice that  $f_j$ have compact supports.

For each $t\in[0,\infty)$ define
\[
	\Phi_{j,t}\ := \ (1+tf_j)\Phi_j,
\]
and set
\[
	P_{j,t}\ :=\ P_j+tf_j\Phi_j\ = \ D_j+(1+tf_j)\Phi_j\ =\ D_j+\Phi_{j,t}.
\]

\begin{lemma}\label{L:Callias-perturbation}
For $j=1,2$, we have 
\begin{enumerate}
\item For every $t\ge 0$, the operator $P_{j,t}=D_j+\Phi_{j,t}:C_0^\infty(M_j,E_j^\pm)\to C_0^\infty(M_j,E_j^\mp)$ is of Callias-type, and, hence, is Fredholm.

\item The exists a constant  $b>0$ and a compact subset $K_{j,b}\Subset M_j$, such that $\Sigma_j\subset M_j\setminus K_{j,b}$ and for every $t\ge b$, the essential support of $P_{j,t}$ is contained in $K_{j,b}$. 
\end{enumerate}
\end{lemma}

Notice, that Lemma~\ref{L:Callias-perturbation}.(\romannumeral2) implies that for large $t$, condition (\romannumeral2) of  Assumption \ref{A:gluing} is satisfied for the operators $P_{j,t}$.

\begin{proof}
(\romannumeral1) Direct computation yields
\begin{eqnarray*}
[D_j,\Phi_{j,t}]_+\!\!\!&=&\!\!\!(1+tf_j)[D_j,\Phi_j]_++\sqrt{-1}t\sigma(D_j)(df_j)\Phi_j, \\
\Phi_{j,t}^2+[D_j,\Phi_{j,t}]_+\!\!\!&=&\!\!\!(1+tf_j)^2\Phi_j^2+(1+tf_j)[D_j,\Phi_j]_++\sqrt{-1}t\sigma(D_j)(df_j)\Phi_j \\
\!\!\!&=&\!\!\!\Phi_j^2+[D_j,\Phi_j]_++(t^2f_j^2+2tf_j)\Phi_j^2+tf_j[D_j,\Phi_j]_++\sqrt{-1}t\sigma(D_j)(df_j)\Phi_j.
\end{eqnarray*}
Since  both $[D_j,\Phi_j]_+$ and $\sigma(D_j)(df_j)\Phi_j$ are bundle maps, so are $[D_j,\Phi_{j,t}]_+$. Suppose $K_j\Subset M_j$ are the essential supports of $P_j$. Since the supports of $tf_j$ and $df_j$ both lie in the compact sets $\overline{W_j}$. So $K_j\cup\overline{W_j}$ is still compact and can serve as the essential supports of $P_{j,t}$. Therefore, $P_{j,t}$ are Callias-type operators and, hence, are Fredholm.

(\romannumeral2) Since $K_j$ are the essential supports of $P_j$, there exist constants $c_j>0$, such that
\[
|(\Phi_j^2+[D_j,\Phi_j]_+)(x_j)|\ge c_j,\qquad\text{for all }x_j\in M_j\setminus K_j.
\]
Since $\overline{V_j}$ are compact sets, $\Phi_j^2+[D_j,\Phi_j]_+$ have finite lower bounds and $\Phi_j^2$ have positive lower bounds on them. Note that on $\overline{V_j}$, $t^2f_j^2\equiv t^2$ and $df_j\equiv0$. One can find $b$ large enough such that for any $t\ge b$,
\[
|(\Phi_{j,t}^2+[D_j,\Phi_{j,t}]_+)(x_j)|\ge c_j,\qquad\text{for all }x_j\in\overline{V_j}.
\]
Now we set $K_{j,b}:=\overline{K_j\setminus V_j}$. It's easy to see that they are still compact sets and are essential supports of $P_{j,t}$ for $t\ge b$. Clearly, $\Sigma_j\not\subset K_{j,b}$. So we are done.
\end{proof}

From this lemma, we see that after the deformations of the operators, $\Sigma_j$ satisfy Assumption \ref{A:gluing}.(\romannumeral2). It remains to prove the following.

\begin{lemma}\label{L:ind-perturbation}
Let $b$ be the positive constant as in last lemma. Then for $j=1,2$,
\begin{equation}\label{E:ind-perturbation}
\ind P_{j,b}\ =\ \ind P_j.
\end{equation}
\end{lemma}

\begin{proof}
Using a similar argument as in the proof of Lemma \ref{L:stbind}, for any $t,t'\in[0,\infty)$,  $P_{j,t}-P_{j,t'}=(t-t')f_j\Phi_j$ are bounded operators depending continuously on $t-t'\in\RR$. By the stability of the index of a Fredholm operator, $\ind P_{j,t}$ are independent of $t$. Then the lemma follows from setting $t=b$ and $t=0$.
\end{proof}

\subsection{Proof of Theorem \ref{T:relindthm}}\label{SS:pf-relindthm}

Applying the construction of Subsection~\ref{SS:setting} to the operators $P_{1,b}$ and $P_{2,b}$ we obtain operators $P_{3,b}$ and $P_{4,b}$ on $M_3$ and $M_4$ respectively. By Lemma~\ref{L:ind-perturbation} the indexes of these operators are equal to the indexes of $P_3$ and $P_4$ respectively. It follows that it is enough to prove Theorem~\ref{T:relindthm} for operators $P_{j,b}$, $j=1,\ldots,4$. In other words it is enough to prove the theorem for the case when $\Sigma$ satisfies  Assumption \ref{A:gluing}.(\romannumeral2). Then we can apply Corollary \ref{C:additivity}.

From now on we assume that $\Sigma$ satisfies Assumption \ref{A:gluing}.(\romannumeral2) for operators $P_1$ and $P_2$. 

As in Section \ref{S:gluing}, we can define operators $P_{\Sigma_j,j},\ j=1,2$.  Let $P_j',P_j''$ be the restrictions of $P_{\Sigma_j,j}$ to $M_j',M_j''$. By Corollary \ref{C:additivity},
\[
\ind P_1\ =\ \ind P_1'+\ind P_1'',
\]
\[
\ind P_2\ =\ \ind P_2'+\ind P_2''.
\]
Similarly, we also have
\[
\ind P_3\ =\ \ind P_1'+\ind P_2'',
\]
\[
\ind P_4\ =\ \ind P_2'+\ind P_1''.
\]
Combining these four equations, we get 
\[
\ind P_1+\ind P_2\ =\ \ind P_3+\ind P_4
\]
and complete the proof.\hfill $\square$


\end{document}